\documentclass[a4paper,oneside,11pt]{article}%
\usepackage{makeidx}
\usepackage[english]{babel}
\usepackage{amsmath}
\usepackage{amsfonts}
\usepackage{amssymb}
\usepackage{stmaryrd}
\usepackage{graphicx}
\usepackage{mathrsfs}
\usepackage[colorlinks,linkcolor=red,anchorcolor=blue,citecolor=blue,urlcolor=blue]{hyperref}
\usepackage[symbol*,ragged]{footmisc}
\providecommand{\U}[1]{\protect\rule{.1in}{.1in}}
\providecommand{\U}[1]{\protect \rule{.1in}{.1in}}

\newtheorem{theorem}{Theorem}[section]

\newtheorem{corollary}[theorem]{Corollary}

\newtheorem{lemma}[theorem]{Lemma}

\newenvironment{proof}[1][Proof]{\noindent \textbf{#1.} }{\  \rule{0.5em}{0.5em}}
\numberwithin{equation}{section}

\begin{document}

\title{On the Distribution of Cube--Free Numbers \\ with the Form $[n^c]$ }

\author{Min Zhang\footnotemark[1]\,\,\,\, \, \& \,\,Jinjiang Li\footnotemark[2] \vspace*{-4mm} \\
$\textrm{\small Department of Mathematics, China University of Mining and Technology}^{*\,\dag}$
                    \vspace*{-4mm} \\
     \small  Beijing 100083, P. R. China  }

\footnotetext[2]{Corresponding author. \\
    \quad\,\, \textit{ E-mail addresses}:
     \href{mailto:min.zhang.math@gmail.com}{min.zhang.math@gmail.com} (M. Zhang),
     \href{mailto:jinjiang.li.math@gmail.com}{jinjiang.li.math@gmail.com} (J. Li).   }

\date{}
\maketitle


{\textbf{Abstract}}: In this paper, we proved that there are infinite cube--free numbers of the form $[n^c]$ for any fixed real number $1<c<11/6$.


{\textbf{MR(2010) Subject Classification}}: 11N37

\section{Introduction and main result}
Let $\mathbb{N}$ denote the set of natural numbers, $\mathscr{P}$ denote the set of all primes and $\mathfrak{F}_k$ denote the set of $k$-free numbers.
For any real number $c>0$ and any infinite subset $\mathcal{A}$ of $\mathbb{N}$, define
\begin{equation*}
   \mathcal{A}^c:=\big\{n\in\mathbb{N}:\,n=[m^c],\,m\in\mathcal{A}\big\}.
\end{equation*}

Let $\mathcal{A},\,\mathcal{B}$ be two infinite subset of $\mathbb{N}$ and $c$ be a fixed real number. It is an important problem to investigate that
whether $\mathcal{A}^c\cap\mathcal{B}$ is a infinite subset of $\mathbb{N}$ or not. For this problem, several cases have been studied.

\textbf{Case 1.} $\mathcal{A}=\mathbb{N},\,\mathcal{B}=\mathscr{P}$.

Piatetski--Shapiro~\cite{Piatetski-Shapiro} first studied this case. He proved that the set $\mathbb{N}^c\cap\mathscr{P}$ is infinite for $0<c<12/11$. This result is called
\emph{Piatetski-Shapiro Prime Number Theory}. If $0<c\leqslant1$, this result is the simple corollary of PNT. However, for $1<c<12/11$, the conclusion
is not trivial. Later, the exponent $12/11$ was improved by many authors, for historical literatures the reader should
consult Rivat and Wu~\cite{Rivat-Wu} and its references.

\textbf{Case 2.} $\mathcal{A}=\mathcal{B}=\mathscr{P}$.

 This case is quite difficult, even though $0<c<1$ is not trivial. Balog~\cite{Balog} proved that the set $\mathscr{P}^c\cap\mathscr{P}$ is infinite
 for $0<c<5/6$. Although Balog proved that, for almost all $c>1$, the set $\mathscr{P}^c\cap\mathscr{P}$ is infinite in the sense of Lebesgue measure,
 he can not prove that $\mathscr{P}^c\cap\mathscr{P}$ is infinite for each fixed $c>1$.

\textbf{Case 3.} $\mathcal{A}=\mathbb{N},\,\mathcal{B}=\mathfrak{F}_2$.

Rieger~\cite{Rieger} first investigated this case. He proved that the set $\mathbb{N}^c\cap\mathfrak{F}_2$ is infinite for $1<c<3/2$, which can be derived
from the result of Deshouillers~\cite{Deshouillers}. More precisely, Rieger proved that the asymptotic formula
\begin{equation*}
   \mathbb{N}^c\cap\mathfrak{F}_2(x):=\sum_{\substack{n\leqslant x\\ [n^c]\in\mathfrak{F}_2}}1=\frac{6}{\pi^2}x+O\big(x^{\frac{2c+1}{4}+\varepsilon}\big)
\end{equation*}
holds for $1<c<3/2$. In 1998, Cao and Zhai proved that the asymptotic formula
\begin{equation*}
   \mathbb{N}^c\cap\mathfrak{F}_2(x)=\frac{6}{\pi^2}x+O\big(x^{\frac{36(c+1)}{97}+\varepsilon}\big)
\end{equation*}
holds for $1<c<61/36$. It is important to emphasise that Stux~\cite{Stux} proved that the set $\mathbb{N}^c\cap\mathfrak{F}_2$ is infinite for almost
all $c\in(1,2)$ in the sense of Lebesgue measure. However, his method can not determine the value of $c$ such that $\mathbb{N}^c\cap\mathfrak{F}_2$ is infinite.
In 2008, Cao and Zhai~\cite{Cao-Zhai-2} improved their earlier result in~\cite{Cao-Zhai} and show that, for any fixed $1<c<149/87$, the set $\mathbb{N}^c\cap\mathfrak{F}_2$ is infinite.

  In this paper, we consider the case $\mathcal{A}=\mathbb{N},\,\mathcal{B}=\mathfrak{F}_3$. To be specific, we shall prove that, for a class of infinite
  sets $\mathscr{A}\subseteq\mathbb{N}$, $\mathscr{A}^c\cap\mathfrak{F}_3$ are infinite sets.

Let $c>1$ be a real number and $\mathscr{A}\subseteq\mathbb{N}$ satisfying the following two conditions:

\noindent
(1) For any $\eta>0$, there holds
\begin{equation}\label{condition-1}
   \mathscr{A}(x):=\sum_{\substack{n\leqslant x\\ n\in\mathscr{A}}}1\gg x^{1-\eta};
\end{equation}

\noindent
(2) Let $\eta>0$ be an arbitrary small real number and $x>1$ be a real number. If $\alpha=a/q$ is a rational number
satisfying $(a,q)=1$ and $2\leqslant q\leqslant x^{\eta}$, then there exists positive constant $\delta$, which depends only on $c$, satisfying $\eta\leqslant\delta<1/2$ such that there holds
\begin{equation}\label{condition-2}
   \sum_{\substack{n\leqslant x\\ n\in\mathscr{A}}}e\big(\alpha[n^c]\big)\ll x^{1-\delta}.
\end{equation}

\noindent
 \textbf{Remark} There are many subsets of $\mathbb{N}$ satisfying (\ref{condition-1}) and (\ref{condition-2}). For
 instance, the sets $\mathbb{N},\,\mathscr{P},\,\mathfrak{F}_k\,(k=2,3,\cdots)$, etc.

The main result is the following theorem.
\begin{theorem}\label{Theorem}
   Let $1<c<11/6,\,\gamma=c^{-1}$ and $0<\varepsilon<10^{-10}$ be a sufficiently small constant. Suppose that the set $\mathscr{A}\subseteq\mathbb{N}$ satisfies
the condition (\ref{condition-1}) and (\ref{condition-2}). Then we have
\begin{equation*}
   \mathscr{A}^c\cap\mathfrak{F}_3(x):=\sum_{\substack{n\leqslant x\\ n\in\mathscr{A},\, [n^c]\in\mathfrak{F}_3}}1
     =\frac{1}{\zeta(3)}\mathscr{A}(x)+O\big(x^{1-\varepsilon}\big).
\end{equation*}
\end{theorem}

\begin{corollary}\label{corollary}
    Let $1<c<11/6,\,\gamma=c^{-1}$ and $0<\varepsilon<10^{-10}$ be a sufficiently small constant.  Then we have
  \begin{eqnarray*}
       \mathbb{N}^c\cap\mathfrak{F}_3(x)      & = &  \frac{x}{\zeta(3)}+O(x^{1-\varepsilon}),  \\
       \mathscr{P}^c\cap \mathfrak{F}_3(x)    & = &  \frac{1}{\zeta(3)}\int_2^x\frac{\mathrm{d}u}{\log u}+O(xe^{-c_1\sqrt{\log x}}),  \\
       \mathfrak{F}_3^c\cap\mathfrak{F}_3(x)  & = &  \frac{x}{\zeta^2(3)}+O(x^{1-\varepsilon}),
  \end{eqnarray*}
where $c_1$ is an absolute constant.
\end{corollary}

\noindent
\textbf{Notation} In this paper, we use $[x], {x}$ and $\|x\|$ to denote the integral part of $x$, the fractional part of $x$ and the
distance from $x$ to the nearest integer, respectively; $\mu(n)$ denotes M\"{o}bius function; $e(t)=e^{2\pi it}$; $\psi(x)=x-[x]-1/2$;
$n\sim N$ denotes $N<n\leqslant2N$.

\section{Preliminary Lemmas}

In order to prove Theorem we need the following two lemmas.

\begin{lemma}\label{Vaaler-lemma}
   For any $J\geqslant2$, we have
   \begin{equation*}
    \psi(t)=\sum_{1\leqslant|h|\leqslant J}a(h)e(ht)+O\bigg(\sum_{|h|\leqslant J}b(h)e(ht)\bigg),\quad a(h)\ll\frac{1}{|h|},\,\, b(h)\ll \frac{1}{J}.
   \end{equation*}
\end{lemma}
\begin{proof}
  See pp. 116 of Graham and Kolesnik~\cite{Graham-Kolesnik} or Vaaler~\cite{Vaaler}.
\end{proof}

\begin{lemma}\label{Heath-Brown-lemma}
For any $H\geqslant1$, we have
\begin{equation*}
   \psi(\theta)=-\sum_{0<|h|\leqslant H}\frac{e(\theta h)}{2\pi ih}+\big(g(\theta,H)\big),
\end{equation*}
where
\begin{equation*}
   g(\theta,H)=\min\bigg(1,\frac{1}{H\|\theta\|}\bigg)=\sum_{h=-\infty}^{+\infty}a_he(\theta h)
\end{equation*}
and
\begin{equation*}
   a_h\ll \min\bigg(\frac{\log2H}{H},\frac{1}{|h|},\frac{H}{|h|^2}\bigg).
\end{equation*}
\end{lemma}
\begin{proof}
  See pp. 245 of Heath--Brown~\cite{Heath-Brown} .
\end{proof}

\begin{lemma}\label{Buriev-lemma}
  Let $y$ be not an integer, $\alpha\in(0,1),\,H\geqslant3$. Then we have
  \begin{equation*}
     e(-\alpha\{y\})=\sum_{|h|\leqslant H}c_h(\alpha)e(hy)+O\bigg(\min\Big(1,\frac{1}{H\|y\|}\Big)\bigg),
  \end{equation*}
  where
  \begin{equation*}
    c_h(\alpha):=\frac{1-e(-\alpha)}{2\pi i(h+\alpha)}.
  \end{equation*}
\end{lemma}
\begin{proof}
  See the thesis of Buriev~\cite{Buriev}.
\end{proof}

\section{Proof of Theorem~\ref{Theorem}}

In this section, we shall prove Theorem~\ref{Theorem}. Let $c$ and $\varepsilon$ satisfy the conditions in Theorem~\ref{Theorem}. It is well known that the
characteristic function of cube--free numbers is
\begin{equation*}
\sum_{d^3\mid n}\mu(d)=\left\{
    \begin{array}{lc}
      0, & \exists m\,\,\, \textrm{s.t. $m^3|n$}, \\
      1, & \textrm{others}.
    \end{array}
\right.
\end{equation*}
Then, we can write
\begin{align}\label{A^c-total}
   \mathscr{A}^c\cap\mathfrak{F}_3(x)
   & :=  \sum_{\substack{n\leqslant x\\ n\in\mathscr{A},\, [n^c]\in\mathfrak{F}_3}}1
   =\sum_{\substack{n\leqslant x\\ n\in\mathscr{A}}}\sum_{d^3\mid[n^c]}\mu(d)
              \nonumber \\
   & =\sum_{\substack{n\leqslant x\\ n\in\mathscr{A}}}\sum_{\substack{d^3|[n^c]\\ d\leqslant x^{\varepsilon}}}\mu(d)
      +\sum_{\substack{n\leqslant x\\ n\in\mathscr{A}}}\sum_{\substack{d^3|[n^c]\\ d> x^{\varepsilon}}}\mu(d)
              \nonumber \\
   & =:\Sigma_1+\Sigma_2.
\end{align}
From the formula
\begin{equation*}
  \sum_{h=1}^qe\bigg(\frac{hn}{q}\bigg)=\left\{
     \begin{array}{ll}
         q, & \quad \textrm{if}\,\, q\mid n, \\
         0, & \quad \textrm{if}\,\,q\nmid n,
     \end{array}
  \right.
\end{equation*}
we get
\begin{align}
  \Sigma_1 &=  \sum_{\substack{m\leqslant x\\ m\in\mathscr{A}}} \sum_{d\leqslant x^{\varepsilon}}\frac{\mu(d)}{d^3}
               \sum_{\ell=1}^{d^3}e\bigg(\frac{\ell[m^c]}{d^3}\bigg)
            =   \sum_{d\leqslant x^{\varepsilon}}\frac{\mu(d)}{d^3}\sum_{\ell=1}^{d^3}
               \sum_{\substack{m\leqslant x\\ m\in\mathscr{A}}}e\bigg(\frac{\ell[m^c]}{d^3}\bigg)      \nonumber \\
           &=  \sum_{d\leqslant x^{\varepsilon}} \frac{\mu(d)}{d^3}\mathscr{A}(x)
               + \sum_{2\leqslant d\leqslant x^{\varepsilon}}\frac{\mu(d)}{d^3}\sum_{\ell=1}^{d^3-1}
               \sum_{\substack{m\leqslant x\\ m\in\mathscr{A}}}e\bigg(\frac{\ell[m^c]}{d^3}\bigg) .
\end{align}
Taking $\eta=2\varepsilon$ in (\ref{condition-2}), then there exists $\delta$ satisfying $2\varepsilon\leqslant\delta\leqslant1/2$. From (\ref{condition-2})
we obtain
\begin{equation*}
   \sum_{2\leqslant d\leqslant x^\varepsilon}\frac{\mu(d)}{d^3}\sum_{\ell=1}^{d^3-1}\sum_{\substack{m\leqslant x\\ m\in\mathscr{A}}}
   e\bigg(\frac{\ell[m^c]}{d^3}\bigg)\ll
   x^{1-\delta}\sum_{2\leqslant d\leqslant x^\varepsilon}\frac{d^3-1}{d^3} \ll x^{1-\varepsilon}.
\end{equation*}
It is easy to see that
\begin{equation}\label{sigma-1-partial}
   \sum_{d\leqslant x^{\varepsilon}}\frac{\mu(d)}{d^3}=\sum_{d=1}^\infty\frac{\mu(d)}{d^3}+O(x^{-2\varepsilon})=\frac{1}{\zeta(3)}+O(x^{-2\varepsilon}).
\end{equation}
From (\ref{sigma-1-partial}) and the fact that $\mathscr{A}\subseteq\mathbb{N}$, we get
\begin{equation}\label{sigma-1}
   \Sigma_1=\frac{1}{\zeta(3)}\mathscr{A}(x)+O(x^{1-2\varepsilon}).
\end{equation}

Now we estimate $\Sigma_2$. We have
\begin{align}\label{sigma-2-1}
  \Sigma_2 & = \sum_{\substack{m\leqslant x\\ m\in\mathscr{A}}}\sum_{\substack{d^3\ell\leqslant m^c<d^3\ell+1\\ d>x^{\varepsilon}}}\mu(d)
             \,\, \ll \,\, \sum_{m\leqslant x}\sum_{\substack{d^3\ell-2<m^c\leqslant d^3\ell+2\\ d>x^\varepsilon}}1
                \nonumber \\
     & \ll\sum_{m\leqslant x} \sum_{\substack{(d^3\ell-2)^\gamma<m\leqslant(d^3\ell+2)^\gamma\\ d>x^{\varepsilon}}}1
       \ll \sum_{\substack{d^3\ell\leqslant x^c\\d>x^\varepsilon}}
       \Big([(d^3\ell+2)^\gamma]-[(d^3\ell-2)^\gamma]\Big)
               \nonumber \\
     & \ll (\log x)^2 \sum_{d\sim D}\sum_{\ell\sim L} \Big([(d^3\ell+2)^\gamma]-[(d^3\ell-2)^\gamma]\Big)
\end{align}
for some pair $(D,L)$, where $x^{\varepsilon}\ll D\ll x^{c/3},\,1\ll L\ll x^{c-3\varepsilon},\, D^3L\ll x^c.$

 If $d\sim D,\,\ell\sim L$, then, by Lagrange's mean value theorem, we get
\begin{equation*}
   (d^3\ell+2)^\gamma-(d^3\ell)^\gamma<2\gamma(d^3\ell)^{\gamma-1}<2\gamma(D^3L)^{\gamma-1}<4\gamma(D^3L)^{\gamma-1}
\end{equation*}
and
\begin{equation*}
   (d^3\ell)^\gamma-(d^3\ell-2)^\gamma<2\gamma(d^3\ell-2)^{\gamma-1}<2\gamma(d^3\ell/2)^{\gamma-1}<4\gamma(D^3L)^{\gamma-1}.
\end{equation*}
Therefore, we obtain
\begin{align}\label{sigma-2-2}
  & \qquad \sum_{d\sim D}\sum_{\ell\sim L}\Big([(d^3\ell+2)^\gamma]-[(d^3\ell-2)^\gamma]\Big) \nonumber \\
  &  = \sum_{\substack{d\sim D,\,\ell\sim L\\ (d^3\ell-2)^\gamma<m\leqslant (d^3\ell+2)^\gamma}}1
   \ll \sum_{\substack{d\sim D,\,\ell\sim L\\ (d^3\ell)^\gamma-4\gamma(D^3L)^{\gamma-1}<m\leqslant(d^3\ell)^\gamma+4\gamma(D^3L)^{\gamma-1} }}1
           \nonumber \\
  & =\sum_{d\sim D}\sum_{\ell\sim L}\Big([(d^3\ell)^\gamma+4\gamma(D^3L)^{\gamma-1}]-[(d^3\ell)^\gamma-4\gamma(D^3L)^{\gamma-1}]\Big)
            \nonumber \\
  & \ll (D^3L)^{\gamma-1}DL
        +\left|\sum_{d\sim D}\sum_{\ell\sim L}\psi\big((d^3\ell)^{\gamma}-4\gamma(D^3L)^{\gamma-1}\big)\right|
            \nonumber \\
  &    \qquad +\left|\sum_{d\sim D}\sum_{\ell\sim L}\psi\big((d^3\ell)^{\gamma}+4\gamma(D^3L)^{\gamma-1}\big)\right|
             \nonumber \\
  & =: (D^3L)^{\gamma-1}DL+|\mathcal{T}_+(D,L)|+|\mathcal{T}_-(D,L)|
              \nonumber \\
  & \ll x^{1-2\varepsilon}+|\mathcal{T}_+(D,L)|+|\mathcal{T}_-(D,L)|,
\end{align}
where the last step uses the following estimate
\begin{equation*}
   (D^3L)^{\gamma-1}DL=(D^3L)^{\gamma-1}D^3L\cdot D^{-2}=(D^3L)^{\gamma}\cdot D^{-2}\ll x^{1-2\varepsilon}.
\end{equation*}

From (\ref{A^c-total}), (\ref{sigma-1}), (\ref{sigma-2-1}) and (\ref{sigma-2-2}), we can see that it is sufficient to show
\begin{equation}\label{T+--condition}
  \mathcal{T}_+(D,L)\ll x^{1-2\varepsilon},\quad \mathcal{T}_-(D,L)\ll x^{1-2\varepsilon}.
\end{equation}
Set $N:=D^3L$. We shall prove
\begin{equation}
  \mathcal{T}_+(D,L)\ll N^{\gamma-2\varepsilon},\quad \mathcal{T}_-(D,L)\ll N^{\gamma-2\varepsilon},
\end{equation}
from which we can deduce (\ref{T+--condition}) immediately.
\noindent
If $D\gg N^{(1-\gamma+2\varepsilon)/2}$, then we have
\begin{equation*}
   \mathcal{T}_\pm(D,L)\ll DL=D^3L\cdot D^{-2}=ND^{-2}\ll N^{1-(1-\gamma+2\varepsilon)}=N^{\gamma-2\varepsilon}.
\end{equation*}
Thus, from what follows, we always assume that $D\ll N^{(1-\gamma+2\varepsilon)/2}$ and $DL\geqslant100 N^{\gamma-2\varepsilon}$.

 Taking $J=[D^2LN^{2\varepsilon-\gamma}]$ in Lemma~\ref{Vaaler-lemma}, then we have
\begin{eqnarray}\label{T(D,L)=I+II}
   \mathcal{T}_\pm(D,L) & = &\sum_{d\sim D}\sum_{\ell\sim L}\Bigg(\sum_{1\leqslant|h|\leqslant J}
                   a(h)e\Big(h\big((d^3\ell)^\gamma\pm4\gamma(D^3L)^{\gamma-1}\big)\Big)
                        \nonumber \\
   &  &  \qquad+O\bigg(\sum_{|h|\leqslant J}b(h)e\Big(h\big((d^3\ell)^\gamma\pm4\gamma(D^3L)^{\gamma-1}\big)\Big) \bigg)\Bigg)
                      \nonumber \\
   & = & \textbf{I}+\textbf{II},
\end{eqnarray}
where
\begin{eqnarray}\label{I-estimate}
  \textbf{I} & = & \sum_{1\leqslant|h|\leqslant J}a(h)\sum_{d\sim D}\sum_{\ell\sim L}e\Big(h\big((d^3\ell)^\gamma\pm4\gamma(D^3L)^{\gamma-1}\big)\Big)
                        \nonumber \\
       & \ll & \sum_{1\leqslant h\leqslant J}\frac{1}{h}\Bigg|\sum_{d\sim D}\sum_{\ell\sim L}e\big(h(d^3\ell)^\gamma\big)\Bigg|,
\end{eqnarray}
\begin{eqnarray}\label{II-estimate}
  \textbf{II} & \ll & \sum_{d\sim D}\sum_{\ell\sim L}  \sum_{|h|\leqslant J}b(h)e\Big(h\big((d^3\ell)^\gamma\pm4\gamma(D^3L)^{\gamma-1}\big)\Big)
                          \nonumber \\
  & \ll & \frac{DL}{J}+\sum_{1\leqslant h\leqslant J}|b(h)| \Bigg|\sum_{d\sim D}\sum_{\ell\sim L}e\big(h(d^3\ell)^\gamma\big)\Bigg|
                          \nonumber \\
  & \ll & N^{\gamma-2\varepsilon}+\sum_{1\leqslant h\leqslant J}\frac{1}{J}\Bigg|\sum_{d\sim D}\sum_{\ell\sim L}e\big(h(d^3\ell)^\gamma\big)\Bigg|
                         \nonumber \\
  & \ll & N^{\gamma-2\varepsilon}+\sum_{1\leqslant h\leqslant J}\frac{1}{h}\Bigg|\sum_{d\sim D}\sum_{\ell\sim L}e\big(h(d^3\ell)^\gamma\big)\Bigg|.
\end{eqnarray}
Combining (\ref{T(D,L)=I+II}), (\ref{I-estimate}) and (\ref{II-estimate}), we derive
\begin{eqnarray}
   \mathcal{T}_\pm(D,L) & \ll & N^{\gamma-2\varepsilon}+\sum_{1\leqslant h\leqslant J}\frac{1}{h}
                                 \Bigg|\sum_{d\sim D}\sum_{\ell\sim L}e\big(h(d^3\ell)^\gamma\big)\Bigg|
                                     \nonumber \\
     & \ll &  N^{\gamma-2\varepsilon}+\frac{1}{H}\big|\mathcal{S}(H,D,L)\big|\cdot \log J,
\end{eqnarray}
for some $1\ll H\ll J$, where
\begin{equation*}
   \mathcal{S}(H,D,L):=\sum_{h\sim H}\Bigg|\sum_{d\sim D}\sum_{\ell\sim L}e\big(h(d^3\ell)^\gamma\big)\Bigg|.
\end{equation*}
Therefore, we only need to show
\begin{equation}\label{S(H,D,L)-upper}
   \mathcal{S}(H,D,L)\ll HN^{\gamma-2\varepsilon}.
\end{equation}
Let $F=HN^\gamma$. Thus, we have $N^\gamma\ll F\ll D^2LN^{2\varepsilon}$. Next, in order to prove (\ref{S(H,D,L)-upper}), we shall
consider the following three cases.

\noindent
\textbf{Case 1} If $D\ll N^{2\gamma-1-6\varepsilon}$, we use exponential pair $(1/2,1/2)$ to estimate the inner sum over $\ell$ and apply
trivial estimate to the sum over $h$ and $d$. Thus, we get
\begin{eqnarray*}
   \mathcal{S}(H,D,L) & \ll & \frac{HD}{HD^{3\gamma}L^{\gamma-1}}+HD\big((HD^{3\gamma}L^{\gamma-1})^{1/2}L^{1/2}\big)    \\
    &  =  &  \frac{DL}{N^\gamma}+HD\big((HN^\gamma L^{-1})^{1/2}L^{1/2}\big)      \\
    & \ll &  N^{1-\gamma}+HD(JN^\gamma)^{1/2} \ll N^{1-\gamma}+HD(D^2L)^{1/2}N^{\varepsilon}  \\
    & = &    N^{1-\gamma}+HD^{1/2}(D^3L)^{1/2}N^\varepsilon      \\
    & \ll &  N^{1-\gamma}+HD^{1/2}N^{1/2+\varepsilon} \ll HN^{\gamma-2\varepsilon}.
\end{eqnarray*}

\noindent
\textbf{Case 2} If $N^{2\gamma-1-6\varepsilon}\ll D\ll N^{6\gamma-3-22\varepsilon}$, by Theorem 7 of Cao and Zhai~\cite{Cao-Zhai-Acta-Arith} with
parameters $(M,M_1,M_2)=(L,H,D)$, we obtain
\begin{eqnarray*}
               N^{-\varepsilon}\cdot \mathcal{S}(H,D,L)
   & \ll &  (F^2L^3H^7D^7)^{1/8}+(F^4H^7D^7)^{1/8}+(F^{18}L^{15}H^{54}D^{54})^{1/58}       \\
   &     &  +(F^{35}L^{26}H^{100}D^{100})^{1/108}+(F^{31}L^{24}H^{92}D^{92})^{1/98}        \\
   &     &  +(F^{10}L^{6}H^{27}D^{27})^{1/29}+(F^{111}L^{86}H^{294}D^{294})^{1/336}        \\
   &     &  +(F^{103}L^{74}H^{266}D^{266})^{1/304}+(F^{119}L^{74}H^{294}D^{294})^{1/336}   \\
   &     &  +(F^{80}L^{19}H^{188}D^{188})^{1/200}+(F^{149}L^{34}H^{344}D^{344})^{1/368}    \\
   &     &  +(F^{43}L^{5}H^{94}D^{94})^{1/100}+(F^2LH^6D^6)^{1/6}                          \\
   &     &  +(F^{4}L^{-1}H^8D^8)^{1/8}+F^{-1/2}LHD                                         \\
   & \ll &  HN^{\gamma-3\varepsilon}.
\end{eqnarray*}

\noindent
\textbf{Case 3} If $D\gg N^{6\gamma-3-22\varepsilon}$, by Theorem 3 of Robert and Sargos~\cite{Robert-Sargos} with parameters $(H,N,M)=(H,D,L)$, we
deduce that
\begin{eqnarray*}
           N^{-2\varepsilon}\cdot \mathcal{S}(H,D,L)
  & \ll &  (HDL)\bigg(\Big(\frac{F}{HDL^2}\Big)^{1/4}+\frac{1}{L^{1/2}}+\frac{1}{F}\bigg)   \\
  & \ll &  HN^{\gamma/4}D^{3/4}L^{1/2}+HDL^{1/2}+DLN^{-\gamma}\ll HN^{\gamma-4\varepsilon}
\end{eqnarray*}
under the condition $D\gg N^{(2-3\gamma+16\varepsilon)/3}$. Moreover, by noting the fact that $\gamma>6/11$, there must hold
$N^{(2-3\gamma+16\varepsilon)/3}\ll N^{6\gamma-3-22\varepsilon}$.

Combining the above three cases, we have finished the proof of Theorem~\ref{Theorem}.

\section{Proof of Corollary~\ref{corollary}}
In this section, we shall prove Corollary~\ref{corollary}. Take $\eta=2\varepsilon$ in (\ref{condition-2}).
Suppose that $2\leqslant d\leqslant x^\varepsilon,\, 1\leqslant\ell \leqslant d^2-1$, then there exist a pair of integers $a$ and $q$ satisfying
$2\leqslant q\leqslant x^\eta,\,1\leqslant a\leqslant q-1,\,(a,q)=1$. Denote $\alpha=a/q$ and
\begin{equation*}
   S_c(x;\mathscr{A},\alpha):=\sum_{\substack{n\leqslant x\\ n\in\mathscr{A}}}e\big(\alpha[n^c]\big).
\end{equation*}
In order to prove Corollary~\ref{corollary}, we need to prove the following lemma.
\begin{lemma}\label{suppli-condition}
   Let $1<c<2$ and $0<\eta(c-1)/100$ be a sufficiently small constant, then we have
   \begin{eqnarray}
       S_c(x;\mathbb{N},\alpha)   & \ll & x^{1-(3-c)/7}\log x,   \label{lemma-1}\\
       S_c(x;\mathscr{P},\alpha)  & \ll & x^{1-(5-2c)/90}\log^{19}x  \label{lemma-2}\\
       S_c(x;\mathfrak{F}_3,\alpha)  & \ll & x^{1-(11-4c)/22}\log^2x.  \label{lemma-3}
   \end{eqnarray}
\end{lemma}
\begin{proof}
    We only need to prove (\ref{lemma-3}), since (\ref{lemma-1}) and (\ref{lemma-2}) are from Lemma 2 of \cite{Cao-Zhai-2}. Obviously, it is sufficient
to prove that the following estimate
\begin{equation}
   S^*_c(N;\mathfrak{F}_3,\alpha):=\sum_{\substack{n\sim N\\ n\in \mathfrak{F}_3}} e\big(\alpha[n^c]\big)\ll N^{1-\delta}\log^\varpi N
\end{equation}
holds with $x^{3/4}\ll N \ll x$.

Taking $H=N^\delta,\, \delta=(11-4c)/22$ in Lemma~\ref{Buriev-lemma}, we get
\begin{eqnarray*}
   S_c^*(N;\mathfrak{F}_3,\alpha) & = & \sum_{\substack{n\sim N\\ n\in\mathfrak{F}_3}} e\big(\alpha n^c-\alpha\{n^c\}\big)  \nonumber \\
   & = & \sum_{\substack{n\sim N\\ n\in\mathfrak{F}_3}}e\big(\alpha n^c\big)
         \bigg(\sum_{|h|\leqslant H}c_h(\alpha)e(hn^c)+O\Big(\min\Big(1,\frac{1}{H\| n^c\|}\Big)\Big)\bigg)           \nonumber \\
   & = & \sum_{|h|\leqslant H}c_h(\alpha)\sum_{\substack{n\sim N\\ n\in\mathfrak{F}_3}}e\big((h+\alpha)n^c\big)+
         O\bigg(\sum_{n\sim N}\min\Big(1,\frac{1}{H\|n^c\|}\Big)\bigg).
\end{eqnarray*}
From Lemma~\ref{Heath-Brown-lemma}, we obtain
\begin{eqnarray*}
   \sum_{n\sim N}\min\Big(1,\frac{1}{H\|n^c\|}\Big)
         & = & \sum_{n\sim N}\sum_{h=-\infty}^{+\infty}a(h)e(hn^c)=\sum_{h=-\infty}^{+\infty}a(h)\sum_{n\sim N}e(hn^c)      \\
         & = & N\cdot a(0)+\sum_{\substack{h=-\infty\\ h\not=0}}^{+\infty}a(h)\sum_{n\sim N}e(hn^c)             \\
         & \ll & N|a(0)|+\sum_{h=1}^{\infty}|a(h)|\left|\sum_{n\sim N}e(hn^c)\right|       \\
\end{eqnarray*}
\begin{eqnarray*}
         & \ll & N^{1-\delta}\log N+\sum_{h=1}^{\infty}\min\bigg(\frac{1}{h},\frac{H}{h^2}\bigg)\big(hN^{c-1}\big)^{4/18}N^{11/18}     \\
         & \ll & N^{1-\delta}\log N+\sum_{h\leqslant H}h^{-7/9}N^{(4c+7)/18}+\sum_{h>H}\frac{H}{h^{16/9}}N^{(4c+7)/18}      \\
         & \ll & N^{1-\delta}\log N+N^{(4c+7)/18}H^{2/9}          \\
         & \ll & N^{1-\delta}\log N +N^{(4c+7+4\delta)/18}\ll N^{1-\delta}\log N,
\end{eqnarray*}
where we use the exponential pair $(4/18,11/18)$ to estimate the sum over $n$.

Since $2\leqslant q\leqslant x^\eta,\,1\leqslant a\leqslant q-1,\,(a,q)=1$, then $x^{-\eta}\leqslant\alpha=a/q\leqslant(q-1)/q=1-1/q\leqslant1-x^{-\eta}$.
Thus, for $|h|\leqslant H$, there holds $x^{-\eta}\leqslant|h+\alpha|\ll N^{\delta}$. Noting that $0<\eta<(c-1)/100$, we get
$|h+\alpha|N^{c-1}\gg x^{-\eta}x^{3(c-1)/4}\gg 1$.

Therefore, we deduce that
\begin{eqnarray*}
    \sum_{\substack{n\sim N\\ n\in\mathfrak{F}_3}}e\big((h+\alpha)n^c\big)
    &  = &   \sum_{n\sim N}\bigg(\sum_{m^3|n}\mu(m)\bigg)e\big((h+\alpha)n^c\big)     \\
    & = &  \sum_{N<m^3t\leqslant2N}\mu(m)e\big((h+\alpha)m^{3c}t^c\big)       \\
    & = & \sum_{\substack{N<m^3n\leqslant 2N\\ m\leqslant N^{\delta}}}\mu(m)e\big((h+\alpha)m^{3c}n^c\big)
            +O(N^{1-2\delta})    \\
    & = & \sum_{m\leqslant N^{\delta}}\mu(m)\sum_{\frac{N}{m^3}<n\leqslant\frac{2N}{m^3}}e\big((h+\alpha)m^{3c}n^c\big)+O(N^{1-2\delta}) \\
    & \ll & \sum_{m\leqslant N^{\delta}}\bigg(|h+\alpha|m^{3c}\Big(\frac{N}{m^3}\Big)^{c-1}\bigg)^{4/18}\bigg(\frac{N}{m^3}\bigg)^{11/18}+N^{1-2\delta} \\
    & \ll & N^{(4c+7+4\delta)/18}\bigg(\sum_{m\leqslant n^\delta}m^{-7/6}\bigg)+N^{1-2\delta}   \\
    & \ll & N^{(4c+7+\delta)/18}+N^{1-2\delta}\ll N^{1-\delta},
\end{eqnarray*}
where we use the exponential pair $(4/18,11/18)$ to estimate the sum over $n$. From (26) of Cao and Zhai~\cite{Cao-Zhai-2}, we have
\begin{equation*}
   \sum_{|h|\leqslant H}c_h(\alpha)\sum_{\substack{n\sim N\\ n\in\mathfrak{F}_3}}e\big((h+\alpha)n^c\big)\ll N^{1-\delta}\log N.
\end{equation*}
This completes the proof of Lemma~\ref{suppli-condition}.
\end{proof}

From the three following formulas
\begin{eqnarray*}
 \mathbb{N}(x) & = & x+O(1), \\
 \mathfrak{F}_3(x) & = & \frac{x}{\zeta(3)}+O(x^{1/2+\varepsilon}),  \\
 \mathscr{P}(x) & = & \int_2^x\frac{\mathrm{d}u}{\log u}+O(xe^{-c_0\sqrt{\log x}}),
\end{eqnarray*}
Theorem~\ref{Theorem} and Lemma~\ref{suppli-condition}, we know that Corollary~\ref{corollary} holds.

\bigskip
\bigskip

\textbf{Acknowledgement}

   The authors would like to express the most and the greatest sincere gratitude to Professor Wenguang Zhai for his valuable
advice and constant encouragement.


\begin{thebibliography}{99}

   \bibitem{Balog} A. Balog, \textit{On a variant of the Piatetski-Shapiro prime number theorem},
                               Seminaire de Theorie Analytique des Nombres de Paris, Paris, 1986.

   \bibitem{Buriev} K. Buriev, \textit{Additive problems with prime numbers}, Moscow: Thesis, Moscow University, 1989 (in Russian).


   \bibitem{Cao-Zhai} X. D. Cao \& W. G. Zhai, \textit{The distribution of square--free numbers of the form $[n^c]$}, Journal de theorie nombres de Bordeaux,
                                                         \textbf{10} (1998), 287--299.

   \bibitem{Cao-Zhai-Acta-Arith} X. D. Cao \& W. G. Zhai, \textit{Multiple exponential sums with monomials}, Acta Arith.,
                                                              \textbf{92} (2000), 195--213.

   \bibitem{Cao-Zhai-2} X. D. Cao \& W. G. Zhai, \textit{On the Distribution of Square--Free Numbers of the Form $[n^c]$ (II)},
                                                      Acta Math. Sinica (Chin. Ser.), \textbf{51}(6) (2008), 1187--1194.




   \bibitem{Deshouillers} J. M.  Deshouillers, \textit{Sur la repartition des mombers $[n^c]$ dans les progressions arithmetiques},
                                                 C. R. Acad. Sci. Paris Ser. A, \textbf{277} (1973), 647--650.

   \bibitem{Graham-Kolesnik} S. W. Graham \& G. Kolesnik, \textit{Van der Corput's Method of Exponential Sums}, Cambridge University Press, 1991.

   \bibitem{Heath-Brown} D. R. Heath-Brown, \textit{The Pjatecki\u{\i}-\u{S}apiro prime number theorem}, J. Number Theory, \textbf{16} (1983), 242--266.

   \bibitem{Piatetski-Shapiro} I. I. Piatetski-Shapiro, \textit{On the distribution of prime numbers in sequences of the form $[f(m)]$},
                                                                 Mat. Sb., \textbf{33} (1953), 559--566.

   \bibitem{Rivat-Wu} J. Rivat \& J. Wu, \textit{Prime numbers of the form $[n^c]$}, Glasgow Math. J., \textbf{43} (2001), 237--254.

   \bibitem{Rieger} G. J. Rieger, \textit{Remark on a paper of Stux concerning squarefree numbers in non--linear sequences},
                                             Pacific J. Math., \textbf{78} (1978), 241--242.

   \bibitem{Robert-Sargos} O. Robert \& P. Sargos, \textit{Three--dimensional exponential sums with monomials}, J. Reine Angew. Math.,
                                                            \textbf{591} (2006), 1--20.

   \bibitem{Stux} I. E. Stux, \textit{Distribution of squarefree intrgers in non-linear sequences}, Pacific J. Math., \textbf{75} (1975), 577--584.


   \bibitem{Vaaler} J. D. Vaaler, \textit{Some extremal functions in Fourier analysis},  Bull. Amer. Math. Soc., \textbf{12}(2) (1985), 183--216.
















\end{thebibliography}
\end{document}